\documentclass{amsart}
\usepackage[utf8]{inputenc}
\usepackage{amsmath}
\usepackage{amsthm}
\usepackage{amssymb}

\usepackage{amsrefs}
\usepackage{verbatimbox} 

\usepackage[all]{xy}

\newtheorem{thm}{Theorem}[section]
\newtheorem{prop}[thm]{Proposition}

\theoremstyle{definition}

\theoremstyle{remark}
\newtheorem{rem}{Remark}[thm]

\makeatletter
\def\th@plain{%
  \thm@notefont{}
  \itshape 
}
\def\th@definition{%
  \thm@notefont{}
  \normalfont 
}
\makeatother

\begin{document}

\title{Quantum Mirror Symmetry for Borcea-Voisin Threefolds}
\date{August 2015}
\author{Andrew Schaug}
\address{Department of Mathematics, University of Michigan, Ann Arbor, MI, 48109}
\email{trygve@umich.edu}

\maketitle

\begin{abstract}
Borcea-Voisin threefolds provided some of the first examples of mirror pairs in the Hodge-theoretic sense, but their mirror symmetry at the quantum level have not previously been shown. We prove a Givental-style quantum mirror theorem for certain Borcea-Voisin threefolds: by means of certain birational models, we show that their Gromov-Witten J-functions  are related by a mirror map to solutions of the multi-parameter Picard-Fuchs equations coming from the variation of Hodge structures of their mirror partners.
\end{abstract}

\section{Background}

Mirror symmetry is one of the most influential strands of current algebraic geometry motivated by physics. In string theory, space-time is locally modelled as the product of the four standard dimensions with a Calabi-Yau complex threefold. In many cases these threefolds come in \textit{mirror pairs}, where the Hodge diamond of one is that of its mirror partner rotated by a right angle. From a physical perspective, the observables of the type IIA string theory of the one (which depends on the K\"ahler structure, varying in a moduli space of dimension $h^{1, 1}$) are identical to the observables of the type IIB string theory of its mirror (which relates to the complex structure, varying in a moduli space of dimension $h^{2, 1}$). There is of course more than just homology to be considered in the full physical duality; for the observable physics to be the same, the quantum correlators of the two theories must also be related in some way. These correlators may be packaged into potentials, from which may be derived \textit{J-functions} and \textit{I-functions} respectively. This has several mathematical consequences, some of which were first noticed by the physicists Candelas, de la Ossa, Green and Parks \cite{CDGP}. The Type IIA correlators count classes of worldsheets on the manifold subject to certain conditions. The mathematical formulation of this curve-counting theory is known as Gromov-Witten theory, and these correlators may be expressed in terms of Gromov-Witten invariants. On the B side, the I-function is given as an optimal solution of the \textit{Picard-Fuchs equations} of a family of Calabi-Yau threefolds, combining all of its \textit{periods integrals}.

In the early 1990s, Borcea \cite{B} and Voisin \cite{V} described a class of Calabi-Yau threefolds, given as $(E \times K)/\langle \sigma_E \times \sigma_K\rangle$, where $E$ is an elliptic curve, $K$ is a K3 surface, and $\sigma_E, \sigma_K$ are anti-symplectic involutions acting on them respectively. They showed that these threefolds exhibited Hodge-theoretic mirror symmetry. However, their mirror symmetry at the quantum level has not previously been demonstrated.



\section{The Aim of this Paper}

Topological symmetry for Borcea-Voisin orbifolds is well known, and was in fact one of the first examples of mirror symmetry shown for Calabi-Yau threefolds, after the quintic. A full physical mirror symmetry was proved in genus-0 for the quintic by Givental \cite{Giv}. The aim of this paper is to prove a similar theorem for Borcea-Voisin threefolds. The results may be summarised in the following theorem.

\begin{thm}
\textbf{(A quantum mirror theorem for Borcea-Voisin threefolds).} Let $\mathcal{Y}$ be a Borcea-Voisin threefold $(E \times K)/\langle \sigma \rangle$ with a well-defined twisted hypersurface birational model (as defined below) and a topological mirror Borcea-Voisin threefold $\check{\mathcal{Y}}$. Then $\check{\mathcal{Y}}$ may be given as the fibre over 0 of a 3-parameter family $\check{\mathcal{Y}}_{\psi, \varphi, \chi}.$

 Let $J_{\mathcal{Y}}(z, t)$ be the ambient, genus-0 Gromov-Witten J-functions for  the sectors generated by $E, K$ and $\mathbf{1}_{\sigma}$. Then there is a corresponding I-function $I_{\check{\mathcal{Y}}}(t, z)$ that satisfies certain equations system of Picard-Fuchs equations generated by the 3-parameter family. In particular:
 \begin{enumerate}
\item For $\mathcal{Y} = X(19, 1, 1), \check{\mathcal{Y}} = X(1, 1, 1)$ in Borcea's list, $I_{\check{\mathcal{Y}}}(\psi, \varphi, \chi)$ compiles the solutions for the full Picard-Fuchs system.
\item In all other cases, $I_{\check{\mathcal{Y}}}$ solves at least one Picard-Fuchs equation, as well as 2 other equations generalised from equations from $X(19, 1, 1)$.
\item For $\mathcal{Y} = X(6, 4, 0), \check{\mathcal{Y}} = X(14, 4, 1)$, $I_{X(6,4,0)}(\psi, 0, 0), I_{X(6, 4, 0)}(0, \varphi, 0)$ and $I_{X(6,4,0)}(0, 0, \chi)$ satisfy the one-parameter Picard Fuchs equations for the three 1-parameter families containing the twisted birational model of $X(14,4,0)$ parametrised by $\psi, \varphi, \chi$ respectively.
\end{enumerate}
Furthermore, $I_{\check{\mathcal{Y}}}$ and $J_{\mathcal{Y}}$ are related by
$$J_{\mathcal{Y}}(\tau(\mathbf{t}), z) = \frac{I_{\check{\mathcal{Y}}(t, z)}(\mathbf{t}, z)}{F(\mathbf{t})},$$ 
where $I_{\check{\mathcal{Y}}} = F(t)z + G(t) + O(z^{-1})$ and $\tau(t) = G(t)/F(t)$ is the classical topological mirror map.

Finally, without restricting to the above three sectors, in most cases $I_{\check{\mathcal{Y}}}$ satisfies a certain convenient class of further Picard-Fuchs equations coming from a family of dimension $>3$.
\end{thm}

The restrictions to the theorem are due to the limits of available computational power. It is of course strongly suspected that these methods should generalise entirely for all ambient genus-0 Gromov-Witten invariants and all Borcea-Voisin threefolds with twisted rational models.

\section{Preliminaries}

\subsection{Gromov-Witten theory}

Gromov-Witten theory provides a means to count curves intersecting given homology classes, and to provide a  definition of what this should mean in certain subtle boundary cases. First, we must define stable curves. 

A marked (possibly orbifold) curve $(C, p_1, p_2, \ldots, p_n)$ is \textit{stable} if $C$ is connected, compact, at worst nodal, no marked point is a node, and has finitely many automorphisms which fix the marked points. That is, every genus 0 irreducible component containing at least three marked points and every genus 1 component has at least one marked point.
 
Given an orbifold $\mathcal{Y}$, a map $f:(C, p_1, p_2, \ldots, p_n) \to \mathcal{Y}$ is \textit{stable} if and only if every component of each fibre is stable. Such a map can only be constant on an irreducible component of $C$ if that component is stable. 

There is a well-defined projective moduli stack $\overline{\mathfrak{M}}_{g, n}(\mathcal{Y}, \beta)$ of stable maps to $\mathcal{Y}$, where the source curves are of genus $g$ with $n$ marked points, and the image of the maps lie in the class $\beta \in H_2(\mathcal{Y})$. \cite{CheRu1} We call this a `closure' of the moduli space of maps from the smooth curves, but the `boundary' made up of nodal curves may in fact have larger dimension. Therefore instead of integrating over the fundamental class, we must integrate over a specified \textit{virtual fundamental class} $[\overline{\mathfrak{M}}_{g, n}(\mathcal{Y}, \beta)]^{\mathrm{vir}}$ \cite{FP}.

To an orbifold $\mathcal{Y}$ we associate an \textit{inertial manifold} $I\mathcal{Y} = \coprod_{(g)} Fix(g)$s. This moduli space is endowed with evaluation maps $\mathrm{ev}_i: \overline{\mathfrak{M}}_{g, n}(\mathcal{Y}, \beta) \to I\mathcal{Y}$,  given in the manifold case by $\mathrm{ev}_i:f \mapsto f(p_i)$. In the orbifold case, the inertial orbifold allows us to keep track of which twisted sector the evaluation map sends a marked point to. We also use this to define Chen-Ruan cohomology \cite{CheRu2}, the natural cohomology for orbifolds.

For $\gamma_1, \gamma_2, \ldots \gamma_n \in H_{\mathrm{CR}}^* (\mathcal{Y})$ we define the \textit{Gromov-Witten invariants} $$\langle \gamma_1, \gamma_2, \ldots, \gamma_n \rangle_{g, n}^{\beta}  = \int_{[\overline{\mathfrak{M}}_{g, n}(\mathcal{Y}, \beta)]^{\mathrm{vir}}} \mathrm{ev}_1^*(\gamma_1)\cup \mathrm{ev}_2^*(\gamma_2) \cup \ldots \cup \mathrm{ev}_n^*(\gamma_n). $$
In simple cases this can be thought of the number of curves transversely intersecting specified subvarieties. More generally, we may wish to include tangency conditions. For this, we note that there is a well-defined smooth orbibundle $\mathbb{L}_i \to \overline{\mathfrak{M}}_{g, n}(\mathcal{Y}, \beta)$  whose fibre at each curve $f$ in the moduli space is the cotangent space at $f(p_i)$. We define $\psi_i = c_1(\mathbb{L}_i)$. At orbifold points we must include the multiplicity of the point as an extra factor, so that $\overline{\psi_i} = n\psi_i$ \cite{CheRu1}. Then the \textit{descendant Gromov-Witten invariants} are given by 
$$\langle \tau_{a_1}(\gamma_1),\ldots, \tau_{a_n}(\gamma_n) \rangle_{g, n}^{\beta}  = \int_{[\overline{\mathfrak{M}}_{g, n}(\mathcal{Y}, \beta)]^{\mathrm{vir}}} \prod_{i=1}^n \mathrm{ev}_i^*(\gamma_i)\overline{\psi}^{a_i}. $$

All these invariants may be packaged into  $$ \langle \mathbf{t}, \ldots, \mathbf{t}\rangle_{g, n}^{\beta} = \sum_{a_1, \ldots, a_n \ge 0} \langle \tau_{k_1}(t_{k_1}), \ldots, \tau_{k_2}(t_{k_n})\rangle,$$ for $\sum_{i\ge 0} t_iz^i \in H_{\mathrm{CR}}^*(\mathcal{Y})[[z^{-1}]]$. This is the generating function for the genus-$g$ \textit{descendant potential}

$$\mathcal{F}_{\mathcal{Y}}^g (\mathbf{t}) = \sum_{n \ge 0} \sum_{\beta \in H_{\mathrm{CR}}^2(\mathcal{Y})} \frac{q^d}{n!} \langle \mathbf{t}, \mathbf{t}, \ldots \mathbf{t}\rangle_{g, n}^{\beta},$$ which takes values over the Novikov ring $\mathbb{C}[[ H_2(\mathcal{Y})]]$ (or, in our case, the subring involving only effective classes $\mathbb{C}[[ H_2(\mathcal{Y}) \cap \mathrm{NE}(\mathcal{Y})]]$.)

As in \cite{CheRu1}, we generally specify a degree-two basis $\{ \varphi_{\alpha}\}$ of $ H_{\mathrm{CR}}^*(\mathcal{Y})$, and write $\mathbf{t} = \sum_{k\ge0}\sum_{\alpha} t_k^{\alpha} \varphi_{\alpha} z^k$. 

After making changing variables according to the \textit{dilaton shift}, $$q_1^0 = t_1^0 -1, q_k^i = t_k^i,$$ we define define Givental's \textit{Lagrangian cone} $\{\partial_{\mathbf{q}} \mathcal{F} = \mathbf{p}\}$. Off zero, it is easy to check that this is a Lagrangian submanifold of the whole phase space, and a cone. It is the image of the \textit{Gromov-Witten J-function} $$J(\mathbf{t}, z) = \varphi_0z + \mathbf{t} + \sum_{n \ge 0}\sum_{a \ge0, i \in I} \frac{1}{n!z^{a+1}}\langle \mathbf{t}, \ldots, \mathbf{t},\varphi_i \psi^a\rangle_{0, n+1}^{\circ} \varphi^i.$$  
This is the object of interest in quantum mirror symmetry on the A (Gromov-Witten) side.

\subsection{Picard Fuchs Equations and Gauss-Manin Connections}

We follow the introduction given by Morrison \cite{Mor}. Given a family of Calabi-Yau $n$-folds $\mathcal{M}_{\psi}$, a smoothly varying family of holomorphic $n$-forms $\omega_{\psi}$, and holomorphically varying classes $\gamma_i(\psi)$ which generate the $n$-homology of $M_{\psi}$, the \textit{period integrals} of the family are defined to be $\int_{\gamma_i(\psi)} \omega_{\psi}$. Where $\gamma_i(\psi)$ varies along a non-contractible path in the base of the family, monodromy considerations must be taken into account, so that the period integrals demonstrate non-trivial behaviour. 

We list all the period integrals for each $\gamma_i$ into a vector $$v(z) = (\int_{\gamma_0(\psi)} \omega_{\psi}, \ldots, \int_{\gamma_r(\psi)} \omega_{\psi}).$$  Consider the span of the first $k$ $\psi$-derivatives of $v(z)$. For generic $\psi$, the dimension is constant, and boudnded by $r+1$. Therefore, for some $k$, the first $k$ $\psi$-derivatives are linearly dependent, so that we have $a_j(\psi)$ such that $$\frac{\mathrm{d}^k}{\mathrm{d}\, \psi^k} v(\psi) + \sum_{j=0}^{k-1} a_j(\psi) \frac{\mathrm{d}^j}{\mathrm{d}\, \psi^j} v(\psi) = 0.$$ The Picard-Fuchs equation characterises the period integrals, so it is given by $$\frac{\mathrm{d}^k}{\mathrm{d}\, \psi^k} f(\psi) + \sum_{j=0}^{k-1} a_j(\psi) \frac{\mathrm{d}^j}{\mathrm{d}\, \psi^j} f(\psi) = 0.$$ In order to avoid complications from singularities, it is common to rewrite this as $$(\psi\frac{\mathrm{d}}{\mathrm{d}\, \psi^k})^k f(\psi) + \sum_{j=0}^{k-1} b_j(\psi)  (\psi\frac{\mathrm{d}^j}{\mathrm{d}\, \psi})^j f(\psi) = 0.$$

The full space of complex structures of a given Calabi-Yau threefold $\mathcal{Y}$ is given by $H^{2, 1}(\mathcal{Y})$. Thus the full Picard Fuchs equations are derived from varying Hodge structures within a family of dimension at most $h^{2, 1}$. In the case of the quintic threefold and elliptic curves, $h^{2, 1} = 1$, so we derive one equation in one parameter $\psi$.

In the more general case, including complete intersections and all Borcea-Voisin manifolds, the full story requires a multi-parameter model: in fact the simplest Borcea-Voisin case is a three-parameter model. Consider a Calabi-Yau threefold $M$ given as the quotient of some hypersurface $\{\mathcal{Q} = 0\}/G$ in some weighted projective space $\mathbb{P}^4(\mathbf{q})$.

The local ring $\mathcal{R}_{\mathrm{loc}}$ of such Calabi-Yau quotients may be generated by the degree-$d$ $G$-invariant monomials of the weighted projective space, identified with a subspace of the cohomology ring (in fact, that part induced from the ambient weighted projective space). That is, it has a degree-symmetric graded basis $$\{m_0 = 1; m_1, m_2, \ldots, m_{h^{2, 1}}; m_{h^{2,1}+1}, \ldots, m_{2h_{2, 1}}; m_{2h^{2, 1}+1} \}.$$ It may therefore be given as a quotient $\mathbb{C}[x, y, z, w, v]^G/\mathcal{I},$ where $\mathcal{I}$ is generated as an ideal by relations among the $m_i$.

Varying the complex structure of $M$ along these monomials produces a family $M_{\psi_1, \ldots, \psi_{h^{2,1}}} = \{\mathcal{Q} + \sum_{i=1}^{h^{2, 1}} \psi_i m_i = 0\}/G.$ The Picard-Fuchs equations are then given by the generating relations of $\mathcal{I},$ replacing $m_i$ by $\frac{\partial}{\partial \psi_i},$ and depending on convention changing variables to some appropriate power of $\psi$. Furthermore, often only a cursory examination of the structure of $\mathcal{R}_{\mathrm{loc}}$ is required to determine the number and order of the Picard-Fuchs equations. For example, for a Calabi-Yau threefold in a one-parameter family, its local ring has a basis $\{1, m, m^2, m^3\}$, so that there is one relation expressing $m^4$ as a linear combination of the rest, given a fourth-order equation. For a three-parameter model, there are six possible degree $2d$ products of $m_1, m_2, m_3$, but the degree-2 part of the local ring must have degree 3; therefore there must be three relations of degree 3, giving 3 Picard-Fuchs equations of order 2, as well as other more complicated equations of higher order.

Another understanding was originally provided by Manin \cite{M}. Given a one-parameter family $\pi : V \to B$, the vector bundle $R^n\pi_*\mathbb{C}\otimes \mathcal{O}_B$ over $B$ comes with a natural \textit{Gauss-Manin} connection, $\nabla_{GM}(\psi),$ for $\psi \in B$. Let $\omega_{\psi}$ be a holomorphic form of appropriate degree to define the period integrals. Then relations $$\sum_{k=0}^n f_k(b) \nabla_{GM}(\psi)^k \omega_{\psi} = 0$$ correspond to the Picard-Fuchs equations $\sum_{k=0}^n f_k(\psi) (\frac{\mathrm{d}^k}{\mathrm{d}\, \psi^k} f) = 0,$ satisfied by the period integrals. Then the Picard-Fuchs equations can be seen in fact to be equivalent to the Gauss-Manin connection, in that they define the flat sections. An analogous situation holds for multi-parameter families.

In general, the solutions to the Picard-Fuchs equations (that is, the periods) may be compiled into an optimal solution, a multivariate function of all parameters, known as the $I-function$ of that family. This is the object of interest in quantum mirror symmetry.

For one-parameter families, the Dwork-Griffiths method provides a method to compute the (one) Picard-Fuchs equation. Suppose we have a family of polynomials $\mathcal{Q}_{\psi}(x_0, x_1, \ldots, x_n)$ of degree $d$ and weights $w_0, \ldots, w_n$. Maximal-degree differentials on $\mathbb{P}(w_0, \ldots, w_n)$ are given by $P\Omega/Q$ with matching degrees for the numerator and denominator, where $$\Omega = \sum_{i=0}^n (-1)^i w_jx_j\bigwedge_{j \ne i}\mathrm{d}x_j.$$

Such forms may be associated to linear forms on $H_n(\mathbb{P}(w_0, \ldots, w_n))$ by the map $\gamma \mapsto \int_{\gamma}\frac{P\Omega}{Q^r}.$ This sets up a correspondence between $H^n(\mathbb{P}(w_0, \ldots, w_n))$ and the forms $\frac{P\Omega}{Q^r}$ modulo $J_{\mathcal{Q}_{\psi}}.$

The key to the algorithm is Griffiths' reduction of pole order formula modulo $J_{\mathcal{Q}_{\psi}}.$ Choose arbitrary polynomials $P_i$ of degree $w_i + rd - \sum_{i=1}^n w_i.$ Then the exterior derivative of $$ \sum_{i<j}\frac{w_ix_iP_j - w_jx_jP_i}{Q^r}\mathrm{d}x_0 \wedge \ldots \wedge \hat{\mathrm{d}{x_i}} \wedge \ldots, \wedge \hat{\mathrm{d}{x_j}} \wedge \ldots \wedge \mathrm{d}x_n$$ is $$\frac{r\sum_{i=0}^nP_i \frac{\partial Q}{\partial x_i} \Omega}{Q^{r+1}} - \frac{\sum_{i=0}^n\frac{\partial P_i}{\partial x_i}\Omega}{Q^l}.$$ It follows that $$\frac{r\sum_{i=0}^n{P_i \frac{\partial Q}{\partial x_i} \Omega}}{Q^{r+1}} \cong_{J_{\mathcal{Q}_{\psi}}} \frac{\sum_{i=0}^n\frac{\partial P_i}{\partial x_i}\Omega}{Q^l}.$$ That is, whenever the numerator of a rational $n$-form is in $J_{\mathcal{Q}_{\psi}},$ it is possible to reduce the pole order explicitly. Starting with such an $n$-form and taking its derivatives, we may perform this reduction as many times as needed to determine linear relation among them. This will give the Picard-Fuchs equation.

\subsection{Borcea-Voisin Orbifolds}

Borcea \cite{B} and Voisin \cite{V} constructed a set of Calabi-Yau threefolds, most of which naturally come in mirror pairs. Let $E$ be an elliptic curve, endowed with its natural involution $\sigma_E$, induced by the map $z \mapsto -z$ in the representation $E = \mathbb{C}/\Lambda$ for some lattice $\Lambda$. Let $K$ be a K3-surface endowed with an anti-symplectic involution $\sigma_K$, that is, an involution which induces the map $-I$ on $H^2(K)$. These latter were initially studied by Nikulin in \cite{N}. Then the quotient $\mathcal{Y} = (E\times K)/\langle \sigma_E, \sigma_K\rangle$, or more properly its smooth resolution, is a \textit{Borcea-Voisin manifold}. It is clearly a Calabi-Yau threefold.

The cohomology of $\mathcal{Y}$ depends on the $K$ and $\sigma_K$. The fixpoint set of $\sigma_K$ is given by a union of $N$ curves, whose genera sum to $N'$. It follows that the $H^{*,*}(\mathcal{Y})$ is given by \begin{equation*} H_{CR}^*(\mathcal{Y}) = \begin{tabular}{llllllll}
    & & & 1 & & & \\
   & & 0 & & 0 & & \\
   & 0 &  & $h_{1,1}$ &  & 0 & \\
   1 & & $h_{2, 1}$ & & $h_{2, 1}$ & & 0 \\
   & 0 &  & $h_{1,1}$ &  & 0 & \\
   & & 0 & & 0 & & \\
   & & & 1 & & & \\
 \end{tabular},\end{equation*}
 where $h_{1,1} = 11+5N-N'$, $h_{2, 1} = 11+5N'-N$. We find that for several Nikulin involutions of K3 surfaces whose fixpoint set has $N$ components and whose genera sum to $N'$, there is another with $N'$ components whose genera sum to $N$, \cite{V}. These therefore correspond to mirror Borcea-Voisin manifolds in the Hodge diamond sense, with $N = 0, N' = 0$ and $N = 2, N' = 2$ corresponding to self-mirror manifolds.

Suppose $E = E_f := \{X^2 + f(Y, Z) = 0\} \subset \mathbb{P}(v_0, v_1, v_2)$ and $K = K_g := \{x^2 + g(y, z, w)\} \subset \mathbb{P}(w_0, w_1, w_2, w_3),$ and $\mathrm{gcd}(v_0, w_0) = 1.$ As in \cite{ABS2} define the \textit{twist map} $$T:(X, Y, Z, x, y, z, w) \mapsto ((\frac{x}{X})^{\frac{v_1}{v_0}}Y, (\frac{x}{X})^{\frac{v_2}{v_0}}Z, y, z, w).$$ The image of $E\times K$ under this map is the hypersurface $$\overline{\mathcal{Y}} = \{f(Y, Z) - f(y, z, w) = 0\} \subset \mathbb{P}(w_0v_1, w_0v_2, v_0w_1, v_0w_2, v_0w_3).$$ $T\vert_{E \times K}$ is generically a double cover; the quotient map induces a birational equivalence $ \mathcal{Y} \dashrightarrow \overline{\mathcal{Y}}$ fitting into the commutative diagram 
\begin{displaymath}
    \xymatrix{
        E \times K \ar@{-->}[rrd]^T \ar[d] & &  \\
        \mathcal{Y} = (E \times K)/\mathbb{Z}_2   \ar@{-->}[rr]^{\overline{T}} & &  \overline{\mathcal{Y}}}
\end{displaymath}

If a Borcea-Voisin threefold $(E_f \times K_g)/\mathbb{Z}_2$ has a Borcea-Voisin mirror, then that mirror may be given by $(E_f \times K_{\check{g}})/\mathbb{Z}_2$ in the same way. $E_f$ is self-mirror, and $K_{\check{g}}$ is the mirror K3 surface of $K_g$. A list of  Borcea-Voisin mirror pairs may be found expressed concisely in \cite{ABS1}. We list two examples below. 
$$\begin{tabular}{|l|l|l|l|}
\hline
$g(y, z, w)$ & $K_g$ & $\check{g}(y, z, w)$ & $K_{\check{g}}$ \\
\hline
$y^6+ z^6 + w^{6}$ & $\mathbb{P}(3, 1, 1, 1)[6]$ & $y^5 + yz^5 + zw^6$ & $\mathbb{P}(25, 10, 8, 7)[50]$ \\
$y^5 + z^5 + w^{10}$ & $\mathbb{P}(5, 2, 2, 1)[10]$ & $w^8 + wz^4 + zy^5$ & $\mathbb{P}(16, 5, 7, 4)[32]$\\
\hline
\end{tabular}$$

To each Borcea-Voisin orbifold Borcea \cite{B} associated invariants $(r, a, \delta)$ coming from the Picard lattice. In Borcea's notation, these correspond respectively to Borcea-Voisin orbifolds
$$\begin{tabular}{|l|l|}
\hline
X(1, 1, 1) & X(19, 1, 1)\\
X(6, 4, 0) & X(14, 4, 0)\\
\hline
\end{tabular}$$
For the first mirror pair, a twist map is defined when $E$ is modelled by $\{X^2+Y^4+Z^4 = 0\},$ and for the second when $E$ is modelled by $\{X^2 + Y^3 + Z^6 = 0\}.$

\section{The Gromov-Witten Theory of Borcea-Voisin Varieties}

In \cite{S}, based on work of \cite{CCIT1}, we found the ambient, genus-zero Gromov-Witten I-function of Borcea-Voisin varieties of the form $E = \{X^2 + f(Y, Z) = 0\} \subset \mathbb{P}(v_0, v_1, v_2)$ and $K = \{x^2 + g(y, z, w)\}$, by  means of expressing the ambient spaces $$(\mathcal{P}(v_0, v_1, v_2) \times \mathcal{P}(w_0, w_1, w_2, w_3))/\mathbb{Z}_2$$ as toric stacks in the sense of \cite{BCS}. For the purposes of that paper, we were chiefly interested in the Fermat polynomials.  Here we shall use the same formulation for more general polynomials. 

The Gromov-Witten J-function is found there from an I-function via a `mirror map', but one which is not shown to solve any Picard-Fuchs equations. We restrict to the Chen-Ruan sectors induced by $D_E = \{ Y = 0\},$ $D_K = \{ y = 0\}$ and the Chen-Ruan sector $\mathbf{1}_{\sigma}$ corresponding to the involution $\sigma$, and set
\begin{equation*}
\tilde{q}_E = e^{2t_1+t_2+t_3}, \quad \tilde{q}_K = e^{w_0t_4+w_1t_5+w_2t_6+w_3t_7}, \quad \tilde{q}_{\sigma} = e^{t_1+t_4+2t_8}.
\end{equation*}
Then, from \cite{S}, we have
\begin{align*}
 \mathrm{(1)} \quad & I_{\mathrm{GW}} (\mathcal{Y}) =  z \sum_{c \in \mathbb{N}_0} \sum_{\substack{a, b \in \mathbb{Z}\\ a\ge -c/2\\b \ge -c/w_0}} \tilde{q}_E^{D_E/z+a} \tilde{q}_{K}^{D_K/z+b} \tilde{q}_{\sigma}^c  \times \nonumber \\
& \frac{\Gamma(2D_E/z+1) \Gamma(D_E/z+1)^2 \prod_{i=1}^4\Gamma(w_iD_K/z+1)}{\Gamma(2D_E/z+2a+c+1)\Gamma(D_E/z+a+1)^2\Gamma(w_0D_K/z+w_0b+c+1)}\times \nonumber \\
& \frac{\Gamma(4D_E/z+4a+2c+1)\Gamma(2w_0D_K/z + 2w_0b+2c + 1)}{\prod_{i=1}^3\Gamma(w_iD_K/z+w_ib+1)\Gamma(2c+1)\Gamma(4D_E/z+1)\Gamma(6D_K/z+1)} \nonumber \\
+& \sum_{c \in \frac{1}{2}\mathbb{N}_0\backslash \mathbb{N}_0} \sum_{\substack{a, b \in \mathbb{Z}\\ a\ge -c/2\\b \ge -c/w_0}}  \tilde{q}_E^{D_E/z+a} \tilde{q}_{K}^{D_K/z+b} \tilde{q}_{\sigma}^c \times \nonumber \\
& \frac{\Gamma(2D_E/z+\frac{1}{2}) \Gamma(D_E/z+1)^2 \Gamma(w_0D_K/z+\frac{1}{2})\prod_{i=1}^3\Gamma(w_iD_K/z+1) }{\Gamma(2D_E/z+2a+c+1)\Gamma(D_E/z+a+1)^2\Gamma(w_0D_K/z+3b+c+1)})\times \nonumber \\
 & \frac{\Gamma(4D_E/z+4a+2c+1)\Gamma(2w_0D_K/z + 2w_0b+2c + 1)}{\prod_{i=1}^3\Gamma(w_iD_K/z+w_ib+1)\Gamma(2c+1)\Gamma(4D_E/z+1)\Gamma(2w_0D_K/z+1)}\mathbf{1}_{\sigma}.
\end{align*}

It follows there from results of \cite{CCIT1} that the Gromov Witten J-function is given by
$$J_{\mathcal{Y}}(\tau(\mathbf{t}), z) = \frac{I_{\check{\mathcal{Y}}(t, z)}}{F(t)},$$ 
where $I_{\check{\mathcal{Y}}} = F(t)z + G(t) + O(z^{-1})$ and $\tau(t) = G(t)/F(t)$ is the classical topological mirror map.

In three dimensions, any two minimal models (including Calabi-Yau threefolds by definition) are related by a sequence of flops \cite{Kol}.  Furthermore, quantum cohomology and therefore genus-zero Gromov-Witten invariants are invariant under flops \cite{LiRu} (indeed Gromov-Witten theory is conjectured to be so more generally). It follows that the genus-zero J-functions of $\mathcal{Y}$ and $\overline{\mathcal{Y}}$ are equal.

\section{The Picard-Fuchs equations of Borcea-Voisin Threefolds}

All computations in this section requiring Groebner bases were performed with the aid of Macaulay2. We work with the hypersurface birational model $\overline{\mathcal{Y}}$. For all cases, there are $\sigma$-invariant quasi-homogeneous monomials of the same degree as the defining polynomials of degree $d$ given by $$m_E = Y^2Z^2, \quad m_K = y^2z^2w^2, \quad m_{\sigma} = YZyzw.$$ More generally, there may be $p$ extra independent monomials $m'_i$, depending on $f(Y, Z)$ and $g(y, z, w)$. 
A suitable most general variation of Hodge structures is given by the family of polynomials $$\mathcal{Q}_{\psi, \varphi, \chi, \varphi_1, \ldots, \varphi_k} = f(Y, Z) + g(y, z, w) + \psi Y^2Z^2 + \varphi y^2z^2w^2 + \chi YZyzw + \sum_{i=1}^p \varphi_i m'_i.$$

We seek homogeneous relations between these monomials modulo $J_{\mathcal{Q}}$, in degrees $2d, \; 3d$ and $4d.$

\subsection{The Involution Equation}

For all Borcea-Voisin manifolds for which the twist map is well defined, we have the `involution relation' $m_Em_K = m_{\sigma}^2,$ which may be thought of as characterising the involution. This cedes one Picard-Fuchs equation, which we shall call the \textit{involution equation}: $$\partial_{\psi}\partial_{\varphi}f = \partial_{\chi}^2f.$$ This equation holds for the period integrals of every Borcea-Voisin threefold for which a twist map exists.

\subsection{The full 3-parameter Picard-Fuchs equations for X(1, 1, 1)}

We find the Picard-Fuchs equations for $\overline{\mathcal{Y}}$ when $E = \{X^2 + Y^4 + Z^4 = 0\}$ and $K = \{x^2 + y^6 + z^6 + w^6 = 0\},$ so that under the twist map $$\mathcal{Q}_{\psi, \varphi, \chi} = Y^4 + Z^4 +y^6+z^6+w^6+\psi Y^2Z^2 + \varphi y^2z^2w^2 + \chi YZyzw.$$ We consider the monomials $m_E, m_K, m_{\sigma}$. We must find a complete generating set of independent homogeneous relations between them, up to degree $4d$. First, all calculations must be done modulo the Jacobean ideal $J_{\mathcal{Q}_{\psi}}.$ A Groebner basis for the Jacobean is found, and then all calculations are performed via the normal form. (This Groebner basis calculation is the chief inhibitor in terms of computation time, especially when the defining equations have degrees higher than those examples dealt with in this paper.)

\textbf{Degree 2d.} There are 6 products of the $m_E, m_K, m_{\sigma}$ of degree $2d$. By Poincar\'e duality they span a 3-dimensional subspace of the local ring. There are therefore 3 independent relations. One is of course provided by the involution relation $R_1 :=  -m_{\sigma}^2 - m_Em_K.$ To find the other two, we first express the monomials in normal form with respect to the Gr\"obner basis of $J_{\mathcal{Q}}$, and reduce the resulting coefficient matrix. After a convenient choice of minors, we find the relations
\begin{align*}
R_2 & := 4(4-\psi^2)m_E^2 + 4\psi\chi m_Em_{\sigma} + \chi^2m_{\sigma}^2 = 0,\\
R_3 & := 4(4-\psi^2)m_Em_{\sigma} + 4\psi\chi m_{\sigma}^2 + \chi^2m_Km_{\sigma} = 0.
\end{align*}
The relation for $m_E$ derived from the Dwork-Griffiths method may also be derived from these.

\textbf{Degree 3d.} There are 10 distinct products of $m_E, m_K, m_{\sigma}$ of degree $3d$. By Poincar\'e duality, we require 9 independent relations to cut this down to rank 1; by multiplying the 3 degree-2 relations by $m_E, m_K, m_{\sigma}$ we get 9 relations. Unfortunately, they are not all independent. This can be seen fairly quickly by noting that $m_3R_2-m_1R_3 = 4(4-\psi^2)m_3R_1$. By some manipulation and factorisation, we can express all other monomials in terms of $m_Em_K^2, m_K^3, m_Em_Km_{\sigma}$ which are themselves subject to the further relations
\small{
\begin{align*}
\begin{split}
 & R_4  := \chi(128\psi^3 \varphi^3  - 48\psi^2 \varphi^2\chi^2  + \chi^6  - 512\psi\varphi^3  + 192\varphi^2\chi^2  + 3456\psi^3  - 13824\psi)m_Em_K^2 \\
&  \quad\quad\quad + 4(32\psi^4\varphi^3  - 6\psi^2\varphi\chi^4  + \psi\chi^6  - 256\psi^2\varphi^3  + 24\varphi\chi^4  + 864\psi^4+ 512\varphi^3  \\
& \quad \quad \quad - 6912\psi^2  + 13824)m_Em_Km_{\sigma} = 0,
\end{split}\\
 \begin{split}
& R_5  := 32(\varphi+3)(\varphi^2-3\varphi+9)(32\psi^4\varphi^3  - 6\psi^2\varphi\chi^4  + \psi\chi^6  - 256\psi^2\varphi^3  + 24\varphi\chi^4  + 864\psi^4\\
& \quad\quad\quad  + 512\varphi^3  - 6912\psi^2  + 13824)m_K^3 - \chi^2(1536\psi^4\varphi^4  - 1024\psi^3\varphi^3\chi^2  + 240\psi^2\varphi^2\chi^4 \\
& \quad \quad\quad  - 24\psi\varphi\chi^6  + \chi^8  - 12288\psi^2\varphi^4  + 4096\psi\varphi^3 \chi^2  - 192\varphi^2\chi^4  - 20736\psi^4 \varphi + 3456\psi^3 \chi^2 \\
& \quad \quad \quad  + 24576\varphi^4  + 165888\psi^2\varphi - 13824\psi\chi^2  - 331776\varphi)m_Em_K^2 = 0.
 \end{split}
\end{align*}
}
Similarly to before, the relation for $m_K$ derived from the Griffiths-Dwork method can be derived from these. All of these relations together may be checked to have rank 9, as required.

\textbf{Degree 4d.} We multiply all of the degree-$3d$ relations by $m_E, m_K, m_{\sigma}$ to find 27 relations; the coefficient matrix of these relations in the product monomials of $m_E, m_K, m_{\sigma}$ is of rank 15, which is the number of possible product monomials of degree $4d$. Therefore this is a complete set of monomials and the local ring is indeed 0 in all degrees higher than $3d,$ and we have no further relations. 

Therefore, the full set of Picard-Fuchs operators for the 3-parameter family is given by the following:
\begin{align*}
\mathrm{(i)} \; & \partial_{\chi}^2 - \partial_{\psi}\partial_{\varphi},\\
\mathrm{(ii)}\; &  4(4-\psi^2)\partial_{\psi}^2 + 4\psi\chi \partial_{\psi}\partial_{\chi} + \chi^2\partial_{\chi}^2,\\
\mathrm{(iii)} \; &  4(4-\psi^2)\partial_{\psi}\partial_{\chi} + 4\psi\chi \partial_{\chi}^2 + \chi^2\partial_{\varphi}\partial_{\chi},\\
\begin{split}
 \mathrm{(iv)} \; &\chi(128\psi^3 \varphi^3  - 48\psi^2 \varphi^2\chi^2  + \chi^6  - 512\psi\varphi^3  + 192\varphi^2\chi^2  + 3456\psi^3  - 13824\psi)\partial_{\psi}\partial_{\varphi}^2 \\
&  \quad + 4(32\psi^4\varphi^3  - 6\psi^2\varphi\chi^4  + \psi\chi^6  - 256\psi^2\varphi^3  + 24\varphi\chi^4  + 864\psi^4+ 512\varphi^3  \\
& \quad   - 6912\psi^2  + 13824)\partial_{\psi}\partial_{\varphi}\partial_{\chi},
\end{split}\\
 \begin{split}
\mathrm{(v)} \; & 32(\varphi+3)(\varphi^2-3\varphi+9)(32\psi^4\varphi^3  - 6\psi^2\varphi\chi^4  + \psi\chi^6  - 256\psi^2\varphi^3  + 24\varphi\chi^4  + 864\psi^4\\
& \quad  + 512\varphi^3  - 6912\psi^2  + 13824)\partial_{\varphi}^3 - \chi^2(1536\psi^4\varphi^4  - 1024\psi^3\varphi^3\chi^2  + 240\psi^2\varphi^2\chi^4 \\
& \quad   - 24\psi\varphi\chi^6  + \chi^8  - 12288\psi^2\varphi^4  + 4096\psi\varphi^3 \chi^2  - 192\varphi^2\chi^4  - 20736\psi^4 \varphi \\
& \quad    + 3456\psi^3 \chi^2 + 24576\varphi^4  + 165888\psi^2\varphi - 13824\psi\chi^2  - 331776\varphi)\partial_{\psi}\partial_{\varphi}^2.
 \end{split}
\end{align*}

\subsection{Three one-parameter Picard-Fuchs equations for X(14, 4, 0)}

For the one-parameter families given by  a given invariant monomial $m_i$, the Dwork-Griffiths method may be used to find linear relations between $$\omega_i^{(r)} = \frac{(-1)^r (r-1)! \psi_i^r m_i^{r-1}\Omega}{Q^r}.$$

Varying $m_E$ to give the family $Y^3 + Z^6 + y^5 + z^5 + w^{10} + \psi Y^2Z^2,$ we find the 2nd-order Picard-Fuchs operator
$$\mathrm{(vi)} \quad (4\psi^3 + 27)\partial_{\psi}^2 + (-16\psi^3 + 54)\partial_{\psi} - 2\psi^3.$$

Varying $m_K$ to give the family $Y^3 + Z^6 + y^5 + z^5 + w^{10} + \varphi y^2z^2w^2,$ we find the 3th-order Picard-Fuchs operator
$$\mathrm{(vii)} \quad(16\varphi^5 + 3125)\partial_{\varphi}^3 + (-216\varphi^5 + 28125)\partial_{\varphi}^2 - 312\varphi^5\partial_{\varphi} - 12\varphi^5.$$

Varying $m_{\sigma}$ to give the family $Y^3 + Z^6 + y^5 + z^5 + w^{10} + \chi YZyzw,$ we find the 4th-order Picard-Fuchs operator
\begin{equation*}
\begin{split}
\mathrm{(viii)} & \quad  (\chi^{30}-112100835937500000000)\partial_{\chi}^4 - 40(\chi^{30}+560504179687500000000)\partial_{\chi}^3\\
&   -300 (\chi^{10}   - 6075000)(\chi ^{20}  + 6075000\chi^{10}   + 36905625000000)\partial_{\chi}^2 - 360\chi^{30}\partial_{\chi} -24\chi^{30}.
\end{split}
\end{equation*}

\subsection{Picard-Fuchs equations derived from other multiplication relations}

There is one extra class of Picard-Fuchs equations which will not be difficult to check. Varying by other $\sigma$-invariant monomials corresponding to different Chen-Ruan classes, there are several other simple relations coming from multiplicative identities. For example, in degree 6, if we set $m'_1 = y^6, m'_2 = z^6, m'_3 = y^3z^3,$ we have the relation $m'_1m'_2 = {m'_3}^2,$ giving the Picard-Fuchs equation $\partial_{\varphi_1}\partial_{\varphi_2} = \partial_{\varphi_3}^2.$ 

\section{Proof of Quantum Mirror Symmetry}

We show that the I-functions already shown to transform to the Gromov-Witten J-functions under a mirror map compile solutions to the Picard-Fuchs equations, that is, the period integrals of their respective families. The equations are given in terms of $\psi, \phi, \chi, \varphi_i$. It will be easier to change the variables in the I-function than the Picard-Fuchs equations, by performing the same analytic continuation in \cite{S}, as follows.

After analytic continuation in $a$ we set $$m = - 4D_E/z - 4a - 2c,$$ so that we can rewrite the factors involving $a$ in a sum over $m$ as $$ \frac{-\frac{\pi i}{2}e^{2\pi i(\frac{m}{4} + \frac{c}{2})}}{e^{-2\pi i D_E/z} -e^{2\pi i (\frac{m}{4} + \frac{c}{2})}}\frac{(-1)^m}{\Gamma(m)\Gamma(1-\frac{m}{2})\Gamma(1-\frac{m}{4}-\frac{c}{2})^2}.$$ 

The identity $\Gamma(x)\Gamma(1-x) = \frac{\pi}{\mathrm{sin}(\pi x)}$ allows us to write the second fraction as $$(-1)^m \frac{\Gamma(\frac{m}{2})\mathrm{sin}(\frac{\pi m} {2})\Gamma(\frac{m}{4}+\frac{c}{2})^2\mathrm{sin}(\pi(\frac{m}{4}+ \frac{c}{2}))^2}{\pi^3}.$$

Similarly for the factors involving $b$, we set $n := -2w_0D_K/z - 2w_0b-2c$, and we similarly rewrite these as $$ \frac{-\frac{\pi i}{2}e^{2\pi i(\frac{n}{2w_0} + \frac{c}{2})}}{e^{-2\pi i D_K/z} -e^{2\pi i (\frac{n}{2w_0} + \frac{c}{2})}}\frac{(-1)^n}{\Gamma(n)\Gamma(1-\frac{n}{2})\prod_{i=1}^3\Gamma(1-\frac{w_in}{2w_0}-\frac{c}{2})}.$$ 

Finally, we make the substitution 
\begin{equation*}
\psi = \tilde{q}_E^2, \quad \varphi = \tilde{q}_K^{w_0}, \quad \chi = \tilde{q}_{\sigma}^{1/2}.
\end{equation*}

\begin{prop}
The involution equation holds for all $I_{\mathcal{Y}}(\psi, \varphi, \chi).$
\end{prop}
\begin{proof}
Consider the coefficient of $\tilde{q}_E^{\frac{m}{2}}\tilde{q}_K^{\frac{n}{w_0}}\tilde{q}_{\sigma}^{2c}$ in $\partial_{\chi}^2 I_{\mathcal{Y}}(\psi, \varphi, \chi).$ To retrieve the corresponding coefficient in $\partial_{\psi}\partial_{\varphi} I_{\mathcal{Y}}(\psi, \varphi, \chi)$ we multiply by $\frac{m}{2}\frac{n}{w_0}$, divide by $(2c+1)(2c+2)$, and perform the shift  $$2c \mapsto 2(c+1), \quad \frac{m}{2} \mapsto \frac{m}{2}-1, \quad \frac{n}{w_0} \mapsto \frac{n}{w_0} -1.$$ Note that $\frac{m}{4} + \frac{c}{2}$ and $\frac{n}{2w_0} + \frac{c}{2}$ are preserved under this transformation. Absorbing the extra factors via the identity $x\Gamma(x) = \Gamma(x+1)$, we see that these coefficients are equal. 
\end{proof}

\begin{thm}[A quantum mirror theorem for X(1,1,1)] $I_{X(19,11)}(\psi, \varphi, \chi)$ solves the Picard-Fuchs equations for the three parameter family containing the twisted birational model of $X(1,1,1)$ parametrised by $\psi, \varphi, \chi.$ 
\end{thm}

\begin{proof}
For $I_{X(19,1,1)}(\psi, \varphi, \chi),$ we have $(w_0,w_1,w_2,w_3) = (25, 10, 8, 7).$ We must check the equations for operators (i)-(v), term by term. Equation (i) has already been shown. Via the identity $\xi^2\partial_{\xi}^2 = (\xi\partial_{\xi})^2-\xi\partial_{\xi}$, operator (ii) may be rewritten $$16\partial_{\psi}^2 -4(\psi\partial_{\psi})^2 + 4\psi\partial_{\psi} + 4(\psi\partial_{\psi})(\chi\partial_{\chi}) + (\chi\partial_{\chi})^2 - \chi\partial_{\chi}.$$. Noting that $\psi\partial_{\psi}$ and $\chi\partial_{\chi}$ have the effect of term-wise multiplication by $\frac{m}{2}$ and $2c$ respectively, we may compare the coefficients of $16\partial_{\psi}^2\partial_{\chi}^2 I_{\mathcal{Y}}(\psi, \varphi, \chi)$ and 
 $$(4(\psi\partial_{\psi})^2 - 4\psi\partial_{\psi} - 4(\psi\partial_{\psi})(\chi\partial_{\chi}) - (\chi\partial_{\chi})^2 + \chi\partial_{\chi})I_{\mathcal{Y}}(\psi, \varphi, \chi).$$ The equation corresponding to operator (iii) holds  similarly. (Note that entirely similar arguments for the above equations hold for any Borcea-Voisin threefold with a twist map, where $E$ has weights $(2,1,1).$)

After some labour, it is straightforward to check equations (iv) and (v) in the same way. With dimension considerations, $I_{X(19,1,1)}(\psi, \varphi, \chi)$ compiles these solutions.
\end{proof}
\begin{rem}
In \cite{S} we split the I-function into terms corresponding to $m \mod 4$ and $n \mod 6$. These separate terms are in fact themselves solutions compiling separate period integrals.
\end{rem}

\begin{thm}[A quantum mirror theorem for 1-parameter families containing X(6, 4, 0)]
$I_{X(6,4,0)}(\psi, 0, 0), I_{X(6, 4, 0)}(0, \varphi, 0)$ and $I_{X(6,4,0)}(0, 0, \chi)$ satisfy the Picard Fuchs equations for the three 1-parameter families containing the twisted birational model of $X(14,4,0)$ parametrised by $\psi, \varphi, \chi$ respectively.
\end{thm}
\begin{proof}
In this case $(w_0, w_1, w_2, w_3) = (5, 2, 2, 1).$ Similarly to the proofs of the previous two propositions, it can be checked that $I_{X(6,4,0)}(\psi, 0, 0), I_{X(6, 4, 0)}(0, \varphi, 0)$ and $I_{X(6,4,0)}(0, 0, \chi)$ lie in the kernels of operators (vi), (vii), (viii) respectively.
\end{proof}
\begin{rem}
In general, the I-function found in \cite{S} for each new extra ambient sector included corresponding to the monomial $m'_j$, the I-function acquires a factor of the form $x_j^{k_j}/k!,$ and all occurrences of $\frac{c}{2}$ in (1) are replaced by $\frac{c}{2} + \frac{1}{2}k_js_{j, 4},$ where $s_{j, 4}$ comes from the representation of that sector in the toric stacky structure. For any Picard-Fuchs equation coming from a product relation as in \S5.4, the definition of $s_{j, 4}$ implies that cancellation in fact occurs in exactly the same way as for the involution equation, and that further equation holds.
\end{rem}

\begin{rem}
With more processing power than that at our disposal, it should be possible to use the same procedure to find full sets of Picard-Fuchs equations for \textit{all} Borcea-Voisin threefolds with twisted birational models and verify the quantum mirror theorem in genus zero for \textit{all} ambient sectors. Another method of attack lies in the GKZ systems of \cite{HoLiYa}. We leave this for another paper.
 \end{rem}



\end{document}